\documentclass[letterpaper, 10 pt, conference]{ieeeconf}

\IEEEoverridecommandlockouts
\usepackage{amsmath}
\usepackage{amssymb}
\usepackage{hyperref}
\usepackage[nameinlink]{cleveref}
\usepackage{graphicx}
\usepackage{xcolor}

\newtheorem{theorem}{Theorem}[section]
\newtheorem{definition}[theorem]{Definition}
\newtheorem{lemma}[theorem]{Lemma}
\newtheorem{proposition}[theorem]{Proposition}

\crefname{equation}{}{}

\title{\LARGE \bf
Exponential and practical exponential stability of second-order formation control systems
}

\author{Raik Suttner and Zhiyong Sun
\thanks{Raik Suttner is with the Institute of Mathematics, University of Wuerz\-burg, Germany (email: {\tt\small raik.suttner@uni-wuerzburg.de}).}%
\thanks{Zhiyong Sun is with the Department of Automatic Control, Lund University, Sweden (email: {\tt\small zhiyong.sun@control.lth.se} and {\tt\small sun.zhiyong.cn@gmail.com}).}
}

\begin{document}
\maketitle
\begin{abstract}
We study the problem of distance-based formation shape control for autonomous agents with double-integrator dynamics. Our considerations are focused on exponential stability properties. For second-order formation systems under the standard gradient-based control law, we prove local exponential stability with respect to the total energy by applying Chetaev's trick to the Lyapunov candidate function. We also propose a novel formation control law, which does not require measurements of relative positions but instead measurements of distances. The \emph{distance-only} control law is based on an approximation of symmetric products of vector fields by sinusoidal perturbations. A suitable averaging analysis reveals that the averaged system coincides with the multi-agent system under the standard gradient-based control law. This allows us to prove practical exponential stability for the system under the distance-only control law.
\end{abstract}

\section{Introduction}
In recent years, formation control for autonomous multi-agent systems has seen rapid advances in both theoretical developments and practical applications~\cite{Oh2015}. Among various types of formation control approaches, distance-based formation control is of particular interest due to its reduced usage of global information (such as global coordinate systems) in the implementation. The goal of distance-based formation control is to reach and maintain a desired target formation, which is defined by prescribed inter-agent distances. Many of the recent studies on distance-based formation are focused on single-integrator models (e.g., \cite{chen2017global,pham2018formation,sun2016exponential,park2015formation, park2018distance,chen2019controlling}). In this paper we will consider second-order (i.e., double-integrator) formation control systems. This is motivated by the fact that, for some applications, acceleration-controlled agents provide a more realistic description of the system dynamics from a physical point of view.

The intention of this paper is to establish (practical) exponential stability of second-order distance-based formation control systems. Exponential stability  features certain beneficial properties including the robustness against small perturbations or measurement errors, and therefore has been one of the key topics in the research field of formation control \cite{mou2016undirected,sun2017robustness}. Though exponential stability for single-integrator distance-based formation systems is well understood  \cite{sun2016exponential}, a characterization of exponential stability for second-order formation systems remains a more challenging task.

There are only a few explicit stability results for second-order distance-based formation system in the literature so far. For example, local asymptotic stability for second-order formation systems is proved in \cite{oh2014distance,ramazani2017rigidity}. The proof in~\cite{oh2014distance} is based on a suitable separation of the second-order system into two first-order systems. The results in \cite{ramazani2017rigidity} are derived from an application of LaSalle's invariance principle to the total energy of the system. Moreover, local convergence of combined rigid formation and flocking control is studied in \cite{deghat2016combined}. The paper \cite{sun2017rigid} presents a comprehensive study on system dynamics for second-order rigid formation systems and establishes a connection between single-integrator and double-integrator formation systems to facilitate  stability analysis. An exponential stability result is derived in \cite{sun2017rigid} by means of the center manifold theorem. However, a clear characterization of the convergence is not available. We remark that the exponential stability for second-order formation systems has significant implications such as robustness properties \cite{sun2017robustness}  and rigid motion control \cite{de2018taming}. Though some papers (e.g., \cite{deghat2016combined, sun2017rigid}) attempted to establish the exponential convergence for second-order formation systems, the approaches are either indirect (\cite{sun2017rigid}) or it is difficult to characterize the exponential convergence rate (\cite{deghat2016combined}). The first contribution of this paper is to provide an alternative strategy to prove local exponential stability for second-order formation systems. The proposed analysis is based on Chetaev's trick for the total energy of the multi-agent system. This provides new insights into the convergence process under gradient-based controllers.

The second contribution of this paper is to introduce a \emph{distance-only} formation control method {\footnote{We clarify the differences between \emph{distance-based} formation control and \emph{distance-only} formation control. By following the convention in the literature \cite{Oh2015}, \emph{distance-based} formation means that target formation shapes are defined by distances; however, the distance-based formation control law (usually derived by a gradient approach) often involves relative position measurements. By \emph{distance-only} formation control we mean that both target formation shapes and formation control implementations only use distances.}} for second-order formation systems in order to achieve a desired rigid formation shape. In the standard gradient-based control law, even if the target formation shape is defined by a certain set of inter-agent distances, the implementation of the control law still requires all agents to measure (or communicate) relative positions with respect to their neighbors. The design of a distance-only formation control law is especially challenging because distances contain less information than relative positions. In the recent paper \cite{Suttner2018}, we have developed an approach for distance-only formation control of single-integrator systems. It sill remains open to develop a feasible distance-only formation control law for second-order systems and to determine its stability and convergence properties.

The distance-only formation control law in~\cite{Suttner2018} for single-integrator agents is based on an approximation of Lie bracket directions. The influence of the Lie brackets is revealed by a suitable averaging analysis, which is based on the findings in~\cite{Kurzweil1988,Liu1997}. This approach works well for kinematic systems. For mechanical systems, like double-integrator agents, a different strategy is needed. We show that the averaging technique in \cite{bullo2002averaging} can be used to derive similar results as in~\cite{Suttner2018} for single-integrator agents. For double-integrator agents, the averaged system is determined by symmetric products of vector fields. The symmetric products can be written as certain Lie brackets, which contain the geodesic spray; see~\cite{BulloBook}. To the best of our knowledge, this is the first time that a distance-only approach is developed for second-order formation systems. We prove practical exponential stability under the assumption of infinitesimal rigidity of the target frameworks. Furthermore, the proposed distance-only formation control law for double-integrator agents can be extended to general coordination control systems modelled by second-order dynamics, such as second-order flocking systems \cite{olfati2006flocking} and second-order multi-robot coordination systems \cite{knorn2016overview}. 

The paper is organized as follows. We present basic definitions and notations on graph and rigidity theory in \Cref{sec:infRig}. The problem formulation for distance-based formation shape control is presented in \Cref{sec:problemDescription}. \Cref{sec:gradientControl} shows a detailed proof of exponential stability for second-order formation systems under relative position measurements, while practical exponential stability for distance-only second-order formations is proved in \Cref{sec:controlLaw}. Simulation examples and comparisons are presented in \Cref{sec:simulations}, followed by concluding remarks in \Cref{sec:conclusions}. A detailed averaging analysis for the proof of practical exponential stability is given in the appendix.

\section{Basic definitions and notation}\label{sec:infRig}
Suppose that $V\colon\mathbb{R}^k\to\mathbb{R}$ is a smooth function (by \emph{smooth} we mean \emph{of class~$C^{\infty}$}). For every $p\in\mathbb{R}^k$ and every positive integer~$l$, we denote by $\operatorname{D}^l\!V(p)$ the~$l$th derivative of~$V$ at~$p$, which is an~$l$-multi-linear form on~$\mathbb{R}^k$. We treat all vectors as a column vectors and denote the transposed of~$p\in\mathbb{R}^k$ by~$p^{\top}$. The first and second derivative of~$V$ at $p\in\mathbb{R}^k$ are frequently represented by the gradient vector $\nabla{V}(p)\in\mathbb{R}^k$ and the Hessian matrix $\nabla^2V(p)\in\mathbb{R}^{k\times{k}}$, respectively. This means that we have $\operatorname{D}\!V(p)v=\langle\!\langle\nabla{V}(p),v\rangle\!\rangle$ and $\operatorname{D}^2\!V(p)(v,w)=\langle\!\langle\nabla^2{V}(p)v,w\rangle\!\rangle$ for all $p,v,w\in\mathbb{R}^k$, where $\langle\!\langle\cdot,\cdot\rangle\!\rangle$ denotes the Euclidean inner product. The Euclidean norm is denoted by~$\|\cdot\|$.

An \emph{undirected graph} $\mathcal{G}=(\mathcal{V},\mathcal{E})$ consists of a set $\mathcal{V}=\{1,\ldots,N\}$ together with a nonempty set~$\mathcal{E}$ of two-element subsets of~$\mathcal{V}$. Each element of~$\mathcal{V}$ is referred to as a \emph{vertex} of~$\mathcal{G}$ and each element of~$\mathcal{E}$ is called an \emph{edge} of~$\mathcal{G}$. As an abbreviation, we denote an edge $\{i,j\}\in\mathcal{E}$ simply by~$ij$. A \emph{framework}~$\mathcal{G}(p)$ in~$\mathbb{R}^n$ consists of the undirected graph~$\mathcal{G}$ of~$N$ vertices and a point
\[
p \ = \ (p_1^{\top},\ldots,p_N^{\top})^{\top} \, \in \, \mathbb{R}^n\times\cdots\times\mathbb{R}^n \ = \ \mathbb{R}^{nN}.
\]
Let $M$ denote the number of edges of $\mathcal{G}$. Order the edges in some way and define the so-called \emph{edge map} $f_{\mathcal{G}}\colon\mathbb{R}^{nN}\to\mathbb{R}^{M}$ of~$\mathcal{G}$ by
\begin{equation}\label{eq:edgeMap}
f_{\mathcal{G}}(p) \ := \ \big(\ldots,\|p_j-p_i\|^2,\ldots)_{ij\in\mathcal{E}}^{\top}
\end{equation}
for every $p=(p_1^{\top},\ldots,p_N^{\top})^{\top}\in\mathbb{R}^{nN}$. For the sake of simplicity, we introduce the notion of infinitesimal rigidity only for the case that the number of vertices~$N$ is greater than or equal to the dimension~$n$ of the surrounding space of the framework. This assumption is satisfied in many applications.
\begin{definition}[see, e.g., \cite{Asimow1979}]\label{def:infinitesimalRigidity}
A framework~$\mathcal{G}(p)$ is \emph{infinitesimally rigid} if and only if the rank of the derivative of the edge map~$f_{\mathcal{G}}$ at~$p$ is equal to~$nN-n(n+1)/2$.
\end{definition}
For each edge $ij\in\mathcal{E}$, let $d_{ij}$ be a positive real number. Define $d:=(d_{ij}^2)_{ij\in\mathcal{E}}\in\mathbb{R}^{M}$, where the components of~$d$ are ordered in the same way as the components of~$f_{\mathcal{G}}$. Define a nonnegative smooth function $V_{\mathcal{G},d}\colon\mathbb{R}^{nN}\to\mathbb{R}$ by
\[
V_{\mathcal{G},d}(p) \; := \; \frac{1}{4}\|f_{\mathcal{G}}(p)-d\|^2 \; = \; \frac{1}{4}\sum_{ij\in\mathcal{E}}\big(\|p_j-p_i\|^2-d_{ij}^2\big)^2
\]
for every $p=(p_1^{\top},\ldots,p_n^{\top})^{\top}\in\mathbb{R}^{nN}$. The following estimates for~$V_{\mathcal{G},d}$ are known from~\cite{Suttner2018}.
\begin{proposition}\label{thm:11112017}
Let~$V_{\mathcal{G},d}$ be defined as above.
\begin{enumerate}
	\item\label{thm:11112017:A} For every~$L>0$, there exists~$\alpha_1>0$ such that
	\[
	\|\nabla{V}_{\mathcal{G},d}(p)\|^2 \ \leq \ \alpha_1\,{V}_{\mathcal{G},d}(p)
	\]
	for every~$p\in\mathbb{R}^{nN}$ with~${V}_{\mathcal{G},d}(p)\leq{L}$.
	\item\label{thm:11112017:B} For every~$L>0$ and every integer~$l\geq2$, there exists~$\alpha_l>0$ such that
	\[
	|\operatorname{D}\!^l{V}_{\mathcal{G},d}(p)(v_1,\ldots,v_l)| \ \leq \ \alpha_l\,\|v_1\|\cdots\|v_l\|
	\]
	for every~$p\in\mathbb{R}^{nN}$ with~${V}_{\mathcal{G},d}(p)\leq{L}$ and all $v_1,\ldots,v_l\in\mathbb{R}^{nN}$.
	\item\label{thm:11112017:C} Suppose that for each~$p\in\mathbb{R}^{nN}$ with~$f_{\mathcal{G}}(p)=d$, the framework~$\mathcal{G}(p)$ is infinitesimally rigid. Then, there exist~$L,\alpha_0>0$ such that
	\[
	\|\nabla{V}_{\mathcal{G},d}(p)\|^2 \ \geq \ \alpha_0\,{V}_{\mathcal{G},d}(p)
	\]
	for every~$p\in\mathbb{R}^{nN}$ with~${V}_{\mathcal{G},d}(p)\leq{L}$.
\end{enumerate}
\end{proposition}

\section{Distance-based formation shape control -- problem description}\label{sec:problemDescription}
\begin{figure*}
\centering\hypertarget{Fig1a}{\includegraphics{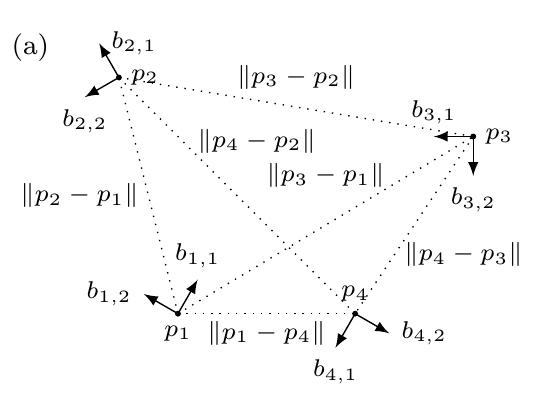}}\qquad\quad\hypertarget{Fig1b}{\includegraphics{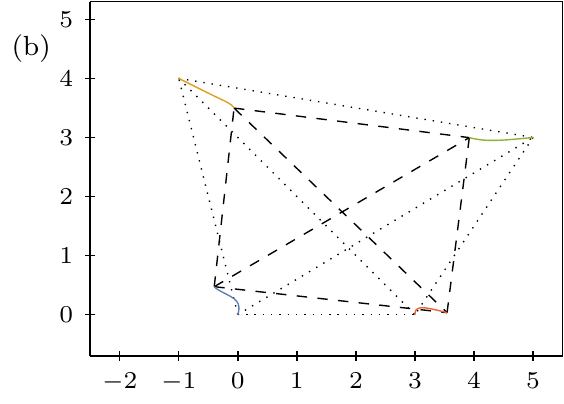}}\qquad\quad\hypertarget{Fig1c}{\includegraphics{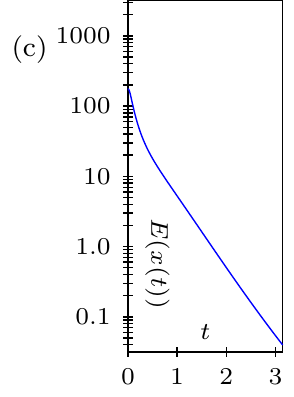}}
\caption{(a): Sketch of a multi-agent system as described in \Cref{sec:problemDescription}. This particular configuration is considered in \Cref{sec:simulations}. (b): Formation shape control of a four-agent system under the gradient-based control law in \Cref{sec:gradientControl}. The initial formation is indicated by dotted lines and the finial formation is indicated by dashed lines. (c): Exponential decay of the total energy during the simulation.}
\label{fig:agentSketch}
\end{figure*}
We consider a system of $N$ point agents in $\mathbb{R}^n$. For each $i\in\{1,\ldots,N\}$, let $b_{i,1},\ldots,b_{i,n}\in\mathbb{R}^n$ be an orthonormal basis of $\mathbb{R}^n$. We assume that the motion of agent $i$ with position~$p_i\in\mathbb{R}^n$ and velocity~$v_i\in\mathbb{R}^n$ is determined by the dynamic equations
\begin{subequations}\label{eq:dynamicPoints}
\begin{align}
\dot{p}_i & \ = \ v_i, \\
\dot{v}_i & \ = \ \sum_{k=1}^{n}a_{i,k}\,b_{i,k},
\end{align}
\end{subequations}
where each of the~$a_{i,k}$ is a real-valued input channel to control the acceleration into direction~$b_{i,k}$. The situation is depicted in \Cref{fig:agentSketch}~\hyperlink{Fig1a}{(a)}. As in \Cref{sec:infRig}, let $\mathcal{G}=(\mathcal{V},\mathcal{E})$ be an undirected graph of $N$ vertices. For each edge $ij\in\mathcal{E}$, let $d_{ij}$ be a positive real number, which is the \emph{desired distance} between agents~$i$ and~$j$. We assume that these distances are \emph{realizable} in~$\mathbb{R}^n$, i.e., the \emph{set}
\begin{equation}\label{eq:setOfDesiredFormations}
\{ (p_1^{\top},\ldots,p_N^{\top})^{\top}\in\mathbb{R}^{nN} \ | \ \forall{ij\in\mathcal{E}}\colon\|p_j-p_i\|=d_{ij} \}
\end{equation}
\emph{of desired formations} is not empty. Note that~\cref{eq:setOfDesiredFormations} coincides with the set of all~$p\in\mathbb{R}^{nN}$ with $f_{\mathcal{G}}(p)=(d_{ij}^2)_{ij\in\mathcal{E}}$, where the edge map~$f_{\mathcal{G}}$ is given by~\cref{eq:edgeMap}. We are interested in a distributed control law that steers the multi-agent system into such a target formation.

The set~\cref{eq:setOfDesiredFormations} of desired formations is defined by the prescribed inter-agent distances~$d_{ij}$. It is assumed that the agents are equipped with sensors so that they can gather information about the other members of the team. In this paper, we consider two different kinds of sensed variables: measurements of relative positions (in \Cref{sec:gradientControl}) and measurements of distances (in \Cref{sec:controlLaw}). In the first case, we assume that agent~$i$ can measure the \emph{relative position} $p_j-p_i$ of agent~$j$ if and only if~$ij$ is an edge of~$\mathcal{G}$. In the second case, we only assume that agent~$i$ can measure the \emph{distance} $\|p_j-p_i\|$ to agent~$j$ if and only if~$ij$ is an edge of~$\mathcal{G}$. In both cases, we assume that each agent can access its own velocity with respect to the individual coordinate system that is determined by the vectors~$b_{i,k}$ in~\cref{eq:dynamicPoints}. In other words, if~$v_i\in\mathbb{R}^n$ is the current velocity of agent~$i$, then we assume that, for every $k\in\{1,\ldots,n\}$, the component $\langle\!\langle{v_i,b_{i,k}}\rangle\!\rangle$ of~$v_i$ with respect to~$b_{i,k}$ are known to agent~$i$.

\section{Exponential stability for double-integrator agents with relative position measurements}\label{sec:gradientControl}
A common approach to stabilize the multi-agent system~\cref{eq:dynamicPoints} around the set~\cref{eq:setOfDesiredFormations} of desired formations is a gradient-based control law as follows. For every $i\in\{1,\ldots,N\}$, agent~$i$ is assigned with a suitable local potential function $V_i\colon\mathbb{R}^{nN}\to\mathbb{R}$. A frequently used choice of the~$V_i$ is
\begin{equation}\label{eq:localPotential}
V_i(p) \ := \ \frac{1}{4}\sum_{j\colon\!{ij}\in\mathcal{E}}\big(\|p_j-p_i\|^2-d_{ij}^2\big)^2.
\end{equation}
For every $i\in\{1,\ldots,N\}$, choose a positive (damping) constants $r_i$. We consider~\cref{eq:dynamicPoints} under the control law
\begin{equation}\label{eq:gradientControlLaw}
a_{i,k} \ = \ -r_i\langle\!\langle{v_i},b_{i,k}\rangle\!\rangle - \langle\!\langle\nabla_{p_i}V_i(p),b_{i,k}\rangle\!\rangle,
\end{equation}
where $\nabla_{p_i}V_i(p)\in\mathbb{R}^n$ denotes the gradient of~$V_i$ with respect to the~$i$th position vector $p_i\in\mathbb{R}^{n}$ at $p=(p_1^{\top},\ldots,p_N^{\top})^{\top}\in\mathbb{R}^{nN}$. Note that an implementation of~\cref{eq:gradientControlLaw} requires measurements of relative positions. When we insert~\cref{eq:gradientControlLaw} into~\cref{eq:dynamicPoints},   we obtain the closed-loop system
\begin{subequations}\label{eq:gradientSystem1}
\begin{align}
\dot{p}_i & \ = \ v_i, \qquad i = 1,\ldots,N, \\
\dot{v}_i & \ = \ -r_i\,v_i - \nabla_{p_i}{V_i}(p).
\end{align}
\end{subequations}
The \emph{total energy}~$E\colon\mathbb{R}^{2nN}\to\mathbb{R}$ of the multi-agent system is the sum
\begin{equation}\label{eq:totalEnergy}
E(x) \ := \ T(v) + V(p)
\end{equation}
of the \emph{kinetic energy}
\begin{equation}\label{eq:kineticEnergy}
T(v) \ := \ \frac{1}{2}\|v\|^2 \ = \ \sum_{i=1}^{N}\frac{1}{2}\|v_i\|^2
\end{equation}
and the \emph{potential energy}
\begin{equation}\label{eq:potentialEnergy}
V(p) \ := \ \frac{1}{4}\,\sum_{ij\in\mathcal{E}}\big(\|p_j-p_i\|^2-d_{ij}^2\big)^2\quad
\end{equation}
for every $x=(p^{\top},v^{\top})^{\top}$ with $p=(p_1^{\top},\ldots,p_N^{\top})^{\top}\in\mathbb{R}^{nN}$ and $v^{\top}=(v_1^{\top},\ldots,v_N^{\top})^{\top}\in\mathbb{R}^{nN}$. Using $\nabla_{p_i}V_i=\nabla_{p_i}V$ for every $i\in\{1,\ldots,N\}$, we can write~\cref{eq:gradientSystem1} equivalently as the (second-order) \emph{gradient system}
\begin{subequations}\label{eq:gradientSystem2}
\begin{align}
\dot{p} & \ = \ v, \\
\dot{v} & \ = \ -R\,v - \nabla{V}(p)
\end{align}
\end{subequations}
with the $nN\times{nN}$ diagonal matrix
\begin{equation}\label{eq:dissipativeForce}
R \ := \ \operatorname{diag}(\underbrace{r_1,\ldots,r_1}_{\text{$n$ times}}, \ \ldots, \ \underbrace{r_N,\ldots,r_N}_{\text{$n$ times}}).\quad
\end{equation}
For every $\varepsilon\geq{0}$, we define a function $E_{\varepsilon}\colon\mathbb{R}^{2nN}\to\mathbb{R}$ by
\begin{equation}\label{eq:ChetaevTrick}
E_{\varepsilon}(x) \ := \ E(x) + \varepsilon\,\langle\!\langle\nabla{V}(p),v\rangle\!\rangle
\end{equation}
for every $x=(p^{\top},v^{\top})^{\top}\in\mathbb{R}^{2nN}$. Considering $E_{\varepsilon}$ with some sufficiently small positive~$\varepsilon$ instead of the total energy~$E=E_0$ is sometimes referred to as \emph{Chetaev's trick} (e.g. in~\cite{BulloBook}), which can be traced back to~\cite{ChetaevBook}. We denote by $\dot{E}_{\varepsilon}\colon\mathbb{R}^{2nN}\to\mathbb{R}$ the derivative of~$E_{\varepsilon}$ along solutions of~\cref{eq:gradientSystem2}, i.e., we let
\begin{align}\label{eq:ChetaevTrickDerivative}
\dot{E}_{\varepsilon}(x) & \ := \ -\langle\!\langle{R\,v},v\rangle\!\rangle + \varepsilon\,\langle\!\langle\nabla^2{V}(p)\,v,v\rangle\!\rangle \nonumber \\
& \qquad - \varepsilon\,\|\nabla{V}(p)\|^2 - \varepsilon\,\langle\!\langle{R\,v},\nabla{V}(p)\rangle\!\rangle
\end{align}
for every $x=(p^{\top},v^{\top})^{\top}\in\mathbb{R}^{2nN}$.
\begin{lemma}\label{thm:ChetaevTrick}
Suppose that for every point~$p$ of~\cref{eq:setOfDesiredFormations}, the framework~$\mathcal{G}(p)$ is infinitesimally rigid. Then, there exist $\varepsilon,L,\gamma_0,\gamma_1,\gamma_2>0$ such that
\begin{align*}
\gamma_0\,E(x) \ \leq \ E_{\varepsilon}(x) & \ \leq \ \gamma_1\,E(x), \\
\dot{E}_{\varepsilon}(x) & \ \leq \ -\gamma_2\,E_{\varepsilon}(x)
\end{align*}
for every $x=(p^{\top},v^{\top})^{\top}\in\mathbb{R}^{2nN}$ with $V(p)\leq{L}$.
\end{lemma}
\begin{proof}
We use the same strategy as in the proof of Theorem~6.45 in~\cite{BulloBook}. By Proposition~\ref{thm:11112017}, there exist $L,\alpha_0,\alpha_1,\alpha_2>0$ such that
\begin{align*}
\alpha_0\,V(p) \ \leq \ \|\nabla{V}(p)\|^2 & \ \leq \ \alpha_1\,V(p), \\
\|\nabla^2{V}(p)\,v\| & \ \leq \ \alpha_2\,\|v\|
\end{align*}
for every $p\in\mathbb{R}^{nN}$ with $V(p)\leq{L}$ and every $v\in\mathbb{R}^{nN}$. Let $r_{\text{min}}>0$ and $r_{\text{max}}>0$ denote the minimum and the maximum of~$r_1,\ldots,r_N$, respectively. Using the Cauchy-Schwarz inequality, we conclude that
\begin{align*}
E_{\varepsilon}(x)  \ \geq \ & \frac{1}{2}\|v\|^2 + \frac{1}{\alpha_1}\,\|\nabla{V}(p)\|^2 - \varepsilon\,\|\nabla{V}(p)\|\,\|v\|, \\
E_{\varepsilon}(x)  \ \leq \ & \frac{1}{2}\|v\|^2 + \frac{1}{\alpha_0}\,\|\nabla{V}(p)\|^2 + \varepsilon\,\|\nabla{V}(p)\|\,\|v\|, \\
\dot{E}_{\varepsilon}(x)  \ \leq \ & -r_{\text{min}}\,\|v\|^2 + \varepsilon\,\alpha_2\,\|v\|^2 \\
& - \varepsilon\,\|\nabla{V}(p)\|^2 + \varepsilon\,r_{\text{max}}\,\|v\|\,\|\nabla{V}(p)\|
\end{align*}
for every $x=(p^{\top},v^{\top})^{\top}\in\mathbb{R}^{2nN}$ with $V(p)\leq{L}$. For every $\varepsilon\geq0$, define the real symmetric $2\times{2}$ matrices
\begin{align*}
O_{\varepsilon} & \ := \ \begin{pmatrix} 1/\alpha_1 & -\varepsilon/2 \\ -\varepsilon/2 & 1/2 \end{pmatrix}, \qquad P_{\varepsilon} \ := \ \begin{pmatrix} 1/\alpha_0 & \varepsilon/2 \\ \varepsilon/2 & 1/2 \end{pmatrix} \allowdisplaybreaks \\
Q_{\varepsilon} & \ := \ \begin{pmatrix}\varepsilon & -\varepsilon\,r_{\text{max}}/2 \\ -\varepsilon\,r_{\text{max}}/2 &  r_{\text{min}}-\varepsilon\,\alpha_2 \end{pmatrix}.
\end{align*}
The diagonal matrices $O_0,P_0$ are obviously positive definite. It is also easy to check that there exists some sufficiently small $\varepsilon>0$ such that $O_{\varepsilon},P_{\varepsilon},Q_{\varepsilon}$ are positive definite. Thus, for each $A\in\{O_0,P_0,O_{\varepsilon},P_{\varepsilon},Q_{\varepsilon}\}$, we can define a norm $\|\cdot\|_A$ on $\mathbb{R}^2$ by $\|w\|_A:=(w^{\top}Aw)^{1/2}$. Then, we have
\begin{align*}
\|w(x)\|_{O_0} \ \leq \ E(x) & \ \leq \ \|w(x)\|_{P_0}, \allowdisplaybreaks \\
\|w(x)\|_{O_\varepsilon} \ \leq \ E_{\varepsilon}(x) & \ \leq \ \|w(x)\|_{P_\varepsilon}, \allowdisplaybreaks \\
\dot{E}_{\varepsilon}(x) & \ \leq \ - \|w(x)\|_{Q_\varepsilon}
\end{align*}
for every $x=(p^{\top},v^{\top})^{\top}\in\mathbb{R}^{2nN}$ with $V(p)\leq{L}$ together with the abbreviation $w(x)=(\|\nabla{V}(p)\|, \|v\|)^{\top}\in\mathbb{R}^2$. Now the claim follows from the well-known fact that all norms on a finite dimensional vector space are equivalent.
\end{proof}

An immediate consequence of Lemma~\ref{thm:ChetaevTrick} is that, under the assumption of infinitesimal rigidity, the set~\cref{eq:setOfDesiredFormations} of desired formations is \emph{locally exponentially stable} for the gradient system~\cref{eq:gradientSystem1} with respect to the total energy.
\begin{theorem}\label{thm:exponentialStability}
Suppose that for every point~$p$ of~\cref{eq:setOfDesiredFormations}, the framework~$\mathcal{G}(p)$ is infinitesimally rigid. Then, there exist $L,\lambda,\mu>0$ such that for every~$x_0\in\mathbb{R}^{2nN}$ with~$E(x_0)\leq{L}$, the maximal solution~$x$ of~\cref{eq:gradientSystem1} with~$x(0)=x_0$ satisfies
\[
E(x(t)) \ \leq \ \lambda\,E(x_0)\,\mathrm{e}^{-\mu{t}}
\]
for every~$t\geq{0}$.
\end{theorem}

An simulation example for the exponential stability of~\cref{eq:dynamicPoints} under~\cref{eq:gradientControlLaw} is presented in parts~\hyperlink{Fig1b}{(b)} and~\hyperlink{Fig1c}{(c)} of \Cref{fig:agentSketch}. The choice of parameters and initial positions for the simulation is described in \Cref{sec:simulations}.

\section{Practical exponential stability for double-integrator agents with distance-only measurements}\label{sec:controlLaw}
\begin{figure*}
\centering$\begin{matrix}\hypertarget{Fig2a}{\includegraphics{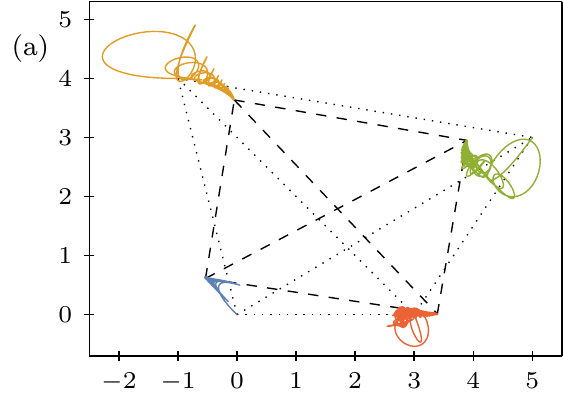}}\qquad&\qquad\hypertarget{Fig2b}{\includegraphics{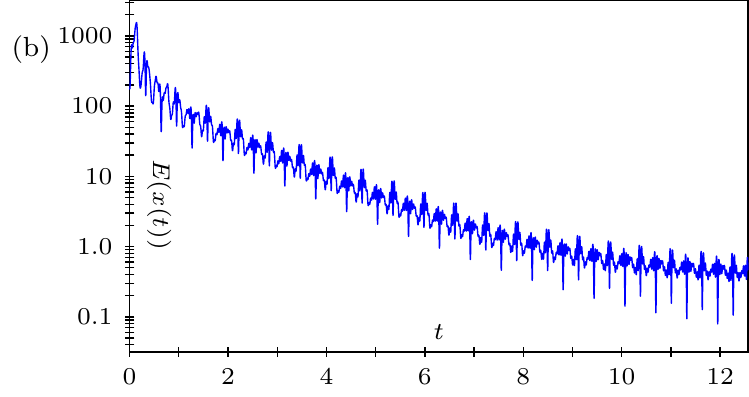}}\\\hypertarget{Fig2c}{\includegraphics{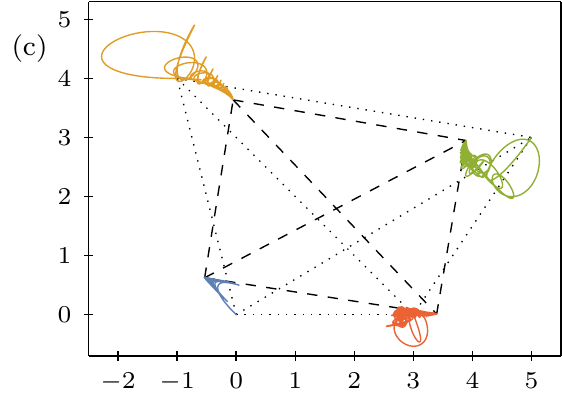}}\qquad&\qquad\hypertarget{Fig2d}{\includegraphics{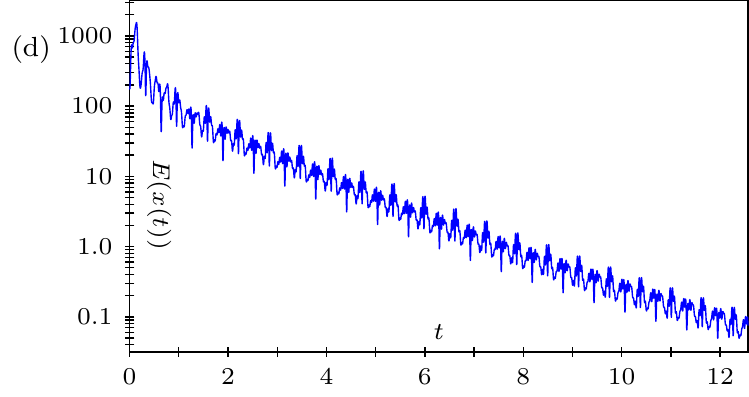}}\end{matrix}$
\caption{(a): Formation shape control of a four-agent system under the distance-only control law in \Cref{sec:controlLaw}. The initial formation is indicated by dotted lines and the finial formation is indicated by dashed lines. (b): Logarithmic plot of the total energy during the simulation. (c), (d): Same situation as in~(a), (b) but the parameters $\rho_i$ in~\cref{eq:controlLaw} are set to zero.}
\label{fig:simulations0102}
\end{figure*}

Our next goal is to reduce the amount of sensed information from relative position vectors~$p_j-p_i$ to the scalar distances $\|p_j-p_i\|$ with $ij\in\mathcal{E}$. For each $i\in\{1,\ldots,N\}$, we choose again the potential function $V_i$ that is given by~\cref{eq:localPotential}. Note agent~$i$ can compute the value of~$V_i$ at any $p\in\mathbb{R}^{nN}$ from measurements of the distances $\|p_j-p_i\|$ with $ij\in\mathcal{E}$. Next, choose~$nN$ pairwise distinct positive real numbers~$\omega_{i,k}$ for $i\in\{1,\ldots,N\}$ and $k\in\{1,\ldots,n\}$. Moreover, for every real number~$\omega>0$, define~$nN$ sinusoids $u^{\omega}_{i,k}\colon\mathbb{R}\to\mathbb{R}$ by
\begin{equation}\label{eq:sinusoids}
u^{\omega}_{i,k}(t) \ := \ 2\,\omega\,\omega_{i,k}\,\cos(\omega\,\omega_{i,k}\,t + \varphi_{i,k}),
\end{equation}
where each $\varphi_{i,k}\in\mathbb{R}$ is an arbitrary but fixed phase shift. Choose arbitrary positive real numbers $r_1,\ldots,r_N$ and $\rho_1,\ldots,\rho_N$. We propose the control law
\begin{equation}\label{eq:controlLaw}
a_{i,k} \ = \ -r_i\,\langle\!\langle{v_i,b_{i,k}}\rangle\!\rangle + u^{\omega}_{i,k}(t)\,\sqrt{\frac{\rho_i}{\omega} + V_i(p)}
\end{equation}
for~\cref{eq:dynamicPoints}, where the real-valued parameter $\omega>0$ has to be chosen sufficiently large (see Theorem~\ref{thm:mainResult} below). By inserting~\cref{eq:controlLaw} into~\cref{eq:dynamicPoints}, we obtain the closed-loop system
\begin{subequations}\label{eq:closedLoopSystem}
\begin{align}
\dot{p}_i & \ = \ v_i, \qquad i = 1,\ldots,N, \\
\dot{v}_i & \ = \ -r_i\,v_i + \sqrt{\frac{\rho_i}{\omega} + V_i(p)}\,u^{\omega}_{i}(t),
\end{align}
\end{subequations}
where the $u^{\omega}_{i}\colon\mathbb{R}\to\mathbb{R}^n$ are defined by
\[
u^{\omega}_{i}(t) \ := \ \sum_{k=1}^{n}u^{\omega}_{i,k}(t)\,b_{i,k}.
\]
\begin{theorem}\label{thm:mainResult}
Suppose that for every point~$p$ of~\cref{eq:setOfDesiredFormations}, the framework~$\mathcal{G}(p)$ is infinitesimally rigid. Then, there exist $\omega_0,\rho,L,\lambda,\mu>0$ such that for every~$\omega\geq\omega_0$, every~$t_0\in\mathbb{R}$, and every~$x_0\in\mathbb{R}^{2nN}$ with~$E(x_0)\leq{L}$, the maximal solution~$x$ of~\cref{eq:closedLoopSystem} with initial condition~$x(t_0)=x_0$ satisfies
\[
E(x(t)) \ \leq \ \sqrt{\frac{\rho}{\omega}} + \lambda\,E(x_0)\,\mathrm{e}^{-\mu(t-t_0)}
\]
for every~$t\geq{t_0}$.
\end{theorem}

A detailed proof of Theorem~\ref{thm:mainResult} is given in the appendix.  At this point, we only indicate why control law~\cref{eq:controlLaw} leads to a decay of the total energy. For this purpose, we write the closed-loop system~\cref{eq:closedLoopSystem} in a suitable control-affine form under open-loop controls. In the first step, we introduce a suitable notation. Recall that, for every $i\in\{1,\ldots,N\}$, the directions $b_{i,1},\ldots,b_{i,n}\in\mathbb{R}^n$ in~\cref{eq:dynamicPoints} are assumed to form an orthonormal basis of~$\mathbb{R}^n$. For every $i\in\{1,\ldots,N\}$ and every $k\in\{1,\ldots,n\}$, define
\[
B_{i,k} \ := \ (0^{\top},\ldots,0^{\top},b_{i,k}^{\top},0^{\top},\ldots,0^{\top})^{\top} \, \in \, \mathbb{R}^{nN},
\]
where~$b_{i,k}$ is at the~$i$th position. It is clear that the vectors~$B_{i,k}$ form an orthonormal basis of~$\mathbb{R}^{nN}$. As an abbreviation, we define an indexing set~$\Lambda$ to be the set of all pairs~$(i,k)$ with $i\in\{1,\ldots,N\}$ and $k\in\{1,\ldots,n\}$. Let $R\in\mathbb{R}^{nN\times{nN}}$ be the diagonal matrix in~\cref{eq:dissipativeForce}. For every $\omega>0$ and every $\ell=(i,k)\in\Lambda$, define a smooth vector field~$f^{\omega}_{\ell}$ on~$\mathbb{R}^{nN}$ by
\begin{equation}\label{eq:newVectorFields}
f^{\omega}_{\ell}(p) \ := \ \sqrt{\frac{\rho_i}{\omega} + V_i(p)}\,B_{\ell}.
\end{equation}
Now we can write the closed-loop system~\cref{eq:closedLoopSystem} equivalently as the second-order control-affine system
\begin{subequations}\label{eq:controlAffineForm}
\begin{align}
\dot{p} & \ = \ v, \label{eq:controlAffineForm:A} \\
\dot{v} & \ = \ -R\,v + \sum_{\ell\in\Lambda}u^{\omega}_{\ell}(t)\,f^{\omega}_{\ell}(p) \label{eq:controlAffineForm:B}
\end{align}
\end{subequations}
with \emph{dissipative force}~$-R\,v$, \emph{open-loop controls}~$u^{\omega}_{\ell}$ and \emph{control vector fields}~$f^{\omega}_{\ell}$, cf.~\cite{BulloBook}.

It is known from~\cite{bullo2002averaging} that the trajectories of a second-order control-affine system of the form~\cref{eq:controlAffineForm} approximate the trajectories of a certain averaged system if the frequency parameter $\omega$ is sufficiently large. The averaged system contains so-called \emph{symmetric products} of the control vector fields (cf.~\cite{bullo2002averaging,BulloBook,Crouch1981}). For our purposes, we only need the following particular case of the symmetric product. For every $\omega>0$ and every $\ell\in\Lambda$, the symmetric product of $f^{\omega}_{\ell}$ with itself is the vector field~$\langle{f^{\omega}_{\ell}\colon\!f^{\omega}_{\ell}}\rangle$ on~$\mathbb{R}^{nN}$ that is defined by
\begin{equation}\label{eq:symmetricProduct}
\langle{f^{\omega}_{\ell}\colon\!f^{\omega}_{\ell}}\rangle(p) \ := \ 2\,\operatorname{D}\!f^{\omega}_{\ell}(p)f^{\omega}_{\ell}(p),
\end{equation}
where $\operatorname{D}\!f^{\omega}_{\ell}(p)$ denotes the derivative of~$f^{\omega}_{\ell}$ at~$p$. A direct computation, using $\nabla_{p_i}V(p)=\nabla_{p_i}V_i(p)$ for every $i\in\{1,\ldots,N\}$, shows that
\begin{equation}\label{eq:symmetricProductGradient}
\sum_{\ell\in\Lambda}\langle{f^{\omega}_{\ell}\colon\!f^{\omega}_{\ell}}\rangle(p) \ = \ \nabla{V}(p)
\end{equation}
for every $p\in\mathbb{R}^{nN}$. By using the averaging techniques from~\cite{bullo2002averaging}, we show in the appendix that the trajectories of~\cref{eq:controlAffineForm} approximate the trajectories of the \emph{averaged system}
\begin{subequations}\label{eq:averagedSystem}
\begin{align}
\dot{p} & \ = \ v, \label{eq:averagedSystem:A} \\
\dot{v} & \ = \ -R\,v - \sum_{\ell\in\Lambda}\langle{f^{\omega}_{\ell}\colon\!f^{\omega}_{\ell}}\rangle(p) \label{eq:averagedSystem:B}
\end{align}
\end{subequations}
for sufficiently large~$\omega>0$. Note that because of~\cref{eq:symmetricProductGradient}, the averaged system~\cref{eq:averagedSystem} coincides with the gradient system~\cref{eq:gradientSystem2}. Moreover, we know from Theorem~\ref{thm:exponentialStability} that the set~\cref{eq:setOfDesiredFormations} of desired formations is locally exponentially stable for~\cref{eq:averagedSystem} with respect to the total energy. By utilizing the exponential stability for~\cref{eq:averagedSystem} and the approximation property for sufficiently large~$\omega$, it is then possible to conclude practical exponential stability for~\cref{eq:controlAffineForm} as it is stated in Theorem~\ref{thm:mainResult}.

On a more intuitive level, one can say that the sinusoidal perturbations in~\cref{eq:controlLaw} allow the agents to explore changes of their local potential functions in a small neighborhood of their current position. In this way they can gather gradient information, which in turn allows an approximation of the gradient-based control~\cref{eq:gradientControlLaw}. The oscillations are required to compensate the reduced amount of information of distance-only measurements.

\section{Simulation examples}\label{sec:simulations}
In this section, we provide simulation results to demonstrate the behavior of~\cref{eq:dynamicPoints} under the distance-only control law~\cref{eq:controlLaw} and to allow a comparison with the behavior of~\cref{eq:dynamicPoints} under the gradient-based control law~\cref{eq:gradientControlLaw}. We consider a system of $N=4$ double-integrator agents in the Euclidean space of dimension $n=2$. For each $i\in\{1,\ldots,4\}$, the coordinate frame of agent~$i$ in~\cref{eq:dynamicPoints} is given by $b_{i,1}=(\cos\phi_i,\sin\phi_i)$ and $b_{i,2}=(-\sin\phi_i,\cos\phi_i)$, where $\phi_i=i\pi/3$; cf.~\Cref{fig:agentSketch}~\hyperlink{Fig1a}{(a)}. We let~$\mathcal{G}$ be the complete graph of~$4$ vertices. This means that each agent can measure the distances to all other members of the team. The common goal of the agents is to reach a rectangular formation with desired distances $d_{12}=d_{34}=3$, $d_{23}=d_{14}=4$, and $d_{13}=d_{24}=5$. By checking the rank condition in Definition~\ref{def:infinitesimalRigidity}, one can verify that, for every point~$p$ of~\cref{eq:setOfDesiredFormations}, the framework $\mathcal{G}(p)$ is infinitesimally rigid. The initial positions are given by $p_1(0)=(0,0)^{\top}$, $p_2(0)=(-1,4)^{\top}$, $p_3(0)=(5,3)^{\top}$, and $p_4(0)=(3,0)^{\top}$. We suppose that the agents rest at the beginning, i.e., their initial velocities are given by $v_i(0)=(0,0)^{\top}$ for every $i\in\{1,\ldots,4\}$. We choose the damping constants $r_i=50$ for every $i\in\{1,\ldots,4\}$. By Theorem~\ref{thm:exponentialStability}, we can expect exponential decay of the total energy under the gradient-based control law~\cref{eq:gradientControlLaw}. This can be verified in the parts~\hyperlink{Fig1b}{(b)} and~\hyperlink{Fig1b}{(c)} of \Cref{fig:agentSketch}.

We also provide data for the case of the distance-only control law~\cref{eq:controlLaw}. For the sinusoids in~\cref{eq:sinusoids}, we choose the frequency coefficients~$\omega_{i,k}=2(i-1)+k$ and the phase shifts~$\varphi_{i,k}=-\pi/2$ for every $i\in\{1,\ldots,4\}$ and every $k\in\{1,2\}$. We choose the offsets $\rho_i=1$ for every $i\in\{1,\ldots,4\}$. The trajectories for $\omega=10$ are shown in parts~\hyperlink{Fig2a}{(a)} and~\hyperlink{Fig2b}{(b)} of \Cref{fig:simulations0102}. We emphasize that an application of Theorem~\ref{thm:mainResult} requires positive offsets $\rho_i$ in~\cref{eq:controlLaw} because the proof of Theorem~\ref{thm:mainResult} exploits a sufficient degree of smoothness. However, it turns out that the performance of control law~\cref{eq:controlLaw} is even better for the nonsmooth case when the $\rho_i$ are all equal to zero. \Cref{fig:simulations0102}~\hyperlink{Fig2d}{(d)} indicates exponential stability. Note that the square root of a nonnegative smooth function is always locally Lipschitz continuous. Thus, even if the $\rho_i$ are all equal to zero, we still have existence and uniqueness of solutions for~\cref{eq:dynamicPoints} under the distance-only control law~\cref{eq:controlLaw}.

\section{Conclusions} \label{sec:conclusions}
In this paper we considered distance-based formation control systems modelled by second-order dynamics for achieving a rigid target shape, and we established their local exponential stability. For the standard gradient-based formation system, by employing the Chetaev's trick we present an explicit analysis of the local exponential stability. We also proposed a distance-only formation control law for stabilizing second-order formation systems. We show by means of a suitable averaging analysis that the trajectories of the second-order distance-only formation system approximate the trajectories of the gradient system. Practical exponential stability is proved for the second-order formation system with distance-only measurements.

\bibliographystyle{ieeetr}
\bibliography{bibFile}

\renewcommand*{\thetheorem}{\Alph{theorem}}
\section*{Appendix: Practical stability analysis}\label{sec:proof}
It remains to prove Theorem~\ref{thm:mainResult}.  We continue in the notation of \Cref{sec:controlLaw}.
\subsection{Velocity transformation}
For every $\omega>0$ and all $\ell,\ell'\in\Lambda$, define sinusoidal functions $U^{\omega}_{\ell},U^{\omega}_{\ell,\ell'}\colon\mathbb{R}\to\mathbb{R}$ by
\begin{align*}
U^{\omega}_{\ell}(t) & \ := \ 2\,\sin(\omega\,\omega_{\ell}\,t + \varphi_{\ell}), \allowdisplaybreaks \\
U^{\omega}_{\ell,\ell'}(t) & \ := \ U^{\omega}_{\ell}(t)\,U^{\omega}_{\ell'}(t) - 2\,\delta_{\ell,\ell'},
\end{align*}
where $\delta_{\ell,\ell'}$ denotes the Kronecker delta of $\ell,\ell'$. Note that each of the $U^{\omega}_{\ell},U^{\omega}_{\ell,\ell'}$ is periodic with zero mean. Moreover, the function~$U^{\omega}_{\ell}$ is an antiderivative of~$u^{\omega}_{\ell}$. For the averaging procedure in the next subsection, it turns out to be convenient to apply the smooth change of variables
\begin{equation}\label{eq:transformedTrajectory}
\tilde{v} \ = \ v - \sum_{\ell\in\Lambda}U^{\omega}_{\ell}(t)\,f^{\omega}_{\ell}(p)
\end{equation}
to the velocities. This change of variables is sometimes referred to as an \emph{L-transformation} (e.g. in~\cite{Baillieul1993}). It is easy to check that, in the variables~\cref{eq:transformedTrajectory}, system~\cref{eq:controlAffineForm} can be written equivalently as
\begin{subequations}\label{eq:transformedControlAffineForm}
\begin{align}
& \dot{p} \ = \ \tilde{v} +  \sum_{\ell\in\Lambda}U^{\omega}_{\ell}(t)\,f^{\omega}_{\ell}(p), \allowdisplaybreaks \\
& \dot{\tilde{v}} \ = \ - R\,\tilde{v} - \nabla{V}(p) - \sum_{\ell\in\Lambda}U^{\omega}_{\ell}(t)\,R\,f^{\omega}_{\ell}(p) \allowdisplaybreaks \\
& \quad - \sum_{\ell\in\Lambda}U^{\omega}_{\ell}(t)\,\operatorname{D}\!f^{\omega}_{\ell}(p)\tilde{v} - \!\sum_{\ell,\ell'\in\Lambda}\!U^{\omega}_{\ell,\ell'}(t)\,\operatorname{D}\!f^{\omega}_{\ell}(p)f^{\omega}_{\ell'}(p),\nonumber
\end{align}
\end{subequations}
where we have used~\cref{eq:symmetricProduct,eq:symmetricProductGradient}. The structure of~\cref{eq:transformedControlAffineForm} already indicates that the averaged equation coincides with the gradient system~\cref{eq:gradientSystem2}. For later references, we state the following result, which quantifies how the change of coordinates affects the value of the total energy~$E$.
\begin{lemma}\label{thm:transformedEnergy}
Let $\rho_{\text{max}}$ denote the maximum of $\rho_{1},\ldots,\rho_{N}$. There exist $\kappa_1,\kappa_2>0$ such that
\begin{align*}
E(\tilde{x}) & \ \leq \ \kappa_1\Big(\frac{\rho_{\text{max}}}{\omega}+E(x)\Big), \\
E(x) & \ \leq \ \kappa_2\Big(\frac{\rho_{\text{max}}}{\omega}+E(\tilde{x})\Big)
\end{align*}
for every $t\in\mathbb{R}$ and every $x=(p^{\top},v^{\top})^{\top}\in\mathbb{R}^{2nN}$, where $\tilde{x}=(p^{\top},\tilde{v}^{\top})^{\top}$ with~$\tilde{v}$ given by~\cref{eq:transformedTrajectory}.
\end{lemma}
\begin{proof}
The estimates follow easily from the definition of the total energy~$E$ in~\cref{eq:totalEnergy} and the change of coordinates~\cref{eq:transformedTrajectory} with the vector fields~$f^{\omega}_{\ell}$ as in~\cref{eq:newVectorFields}.
\end{proof}

\subsection{Averaging}\label{sec:averaging}
We want to apply a modified version of Chetaev's trick to system~\cref{eq:transformedControlAffineForm}. The subsequent averaging result (Proposition~\ref{thm:averaging}) extracts the behavior of the solutions of~\cref{eq:transformedControlAffineForm} in the high-frequency limit. If $g$ is a real-valued function on an interval~$I$, then we write $[g(t)]_{t=t_1}^{t=t_2}:=g(t_2) - g(t_1)$ for $t_1,t_2\in{I}$.
\begin{proposition}\label{thm:averaging}
For fixed $\varepsilon\geq0$, let $E_{\varepsilon}$ and $\dot{E}_{\varepsilon}$ be defined by~\cref{eq:ChetaevTrick} and~\cref{eq:ChetaevTrickDerivative}, respectively. Let $L>0$. Then, there exist constants $\eta_{1},\eta_{2}>0$ and, for every $\omega>0$, smooth real-valued functions $D^{\omega}_1E_{\varepsilon},D^{\omega}_2E_{\varepsilon}$ on $\mathbb{R}^{2nN}$ such that, for every solution~$\tilde{x}\colon{I}\to\mathbb{R}^{2nN}$ of~\cref{eq:transformedControlAffineForm} and all $t_1,t_2\in{I}$, we have
\begin{subequations}\label{eq:integralExpansion}
\begin{align}
& E_{\varepsilon}(\tilde{x}(t_2)) \ = \ E_{\varepsilon}(\tilde{x}(t_1)) + \int_{t_1}^{t_2}\dot{E}_{\varepsilon}(\tilde{x}(t))\,\mathrm{d}t \label{eq:integralExpansion:A} \allowdisplaybreaks \\
& - \big[(D^{\omega}_1E_{\varepsilon})(t,\tilde{x}(t))\big]_{t=t_1}^{t=t_2} + \int_{t_1}^{t_2}(D^{\omega}_2E_{\varepsilon})(t,\tilde{x}(t))\,\mathrm{d}t, \label{eq:integralExpansion:B}
\end{align}
\end{subequations}
and, additionally, the estimates
\[
\big|(D^{\omega}_1E_{\varepsilon})(t,x)\big| \ \leq \ \frac{\eta_1}{\omega} \quad \text{and} \quad \big|(D^{\omega}_2E_{\varepsilon})(t,x)\big| \ \leq \ \frac{\eta_2}{\sqrt{\omega}}
\]
hold for every $\omega\geq1$, every $t\in\mathbb{R}$, and every $x\in\mathbb{R}^{2nN}$ with $E_0(x)\leq{L}$.
\end{proposition}
\begin{proof}
Let $\tilde{x}=(p^{\top},\tilde{v}^{\top})^{\top}\colon{I}\to\mathbb{R}^{2nN}$ be a solution of~\cref{eq:transformedControlAffineForm} and let $t_1,t_2\in{I}$. In the following, sums of the form $\sum_{\ell}$ are meant to run over all elements of $\Lambda$. Using the fundamental theorem of analysis and the defining differential equation~\cref{eq:transformedControlAffineForm} of $\tilde{x}$, the kinetic energy~\cref{eq:kineticEnergy} can be written as
\begin{subequations}\label{eq:kineticEnergyIntegral}
\begin{align}
& T(\tilde{v}(t_2)) \ = \ T(\tilde{v}(t_1)) - \int_{t_1}^{t_2}\langle\!\langle\tilde{v}(t),R\,\tilde{v}(t)\rangle\!\rangle\,\mathrm{d}t \label{eq:kineticEnergyIntegral:A} \allowdisplaybreaks \\
& - \int_{t_1}^{t_2}\langle\!\langle\tilde{v}(t),\nabla{V}(p(t))\rangle\!\rangle\,\mathrm{d}t \label{eq:kineticEnergyIntegral:B} \allowdisplaybreaks \\
& - \sum_{\ell}\int_{t_1}^{t_2}U^{\omega}_{\ell}(t)\,\langle\!\langle\tilde{v}(t),R\,f^{\omega}_{\ell}(p(t))\rangle\!\rangle\,\mathrm{d}t \label{eq:kineticEnergyIntegral:C} \allowdisplaybreaks \\
& - \sum_{\ell}\int_{t_1}^{t_2}U^{\omega}_{\ell}(t)\,\langle\!\langle\tilde{v}(t),\operatorname{D}\!f^{\omega}_{\ell}(p(t))\tilde{v}(t)\rangle\!\rangle\,\mathrm{d}t \label{eq:kineticEnergyIntegral:D} \allowdisplaybreaks \\
& - \sum_{\ell,\ell'}\int_{t_1}^{t_2}U^{\omega}_{\ell,\ell'}(t)\,\langle\!\langle\tilde{v}(t),\operatorname{D}\!f^{\omega}_{\ell}(p(t))f^{\omega}_{\ell'}(p(t))\rangle\!\rangle\,\mathrm{d}t, \label{eq:kineticEnergyIntegral:E}
\end{align}
\end{subequations}
the potential energy~\cref{eq:potentialEnergy} can be written as
\begin{subequations}\label{eq:potentialEnergyIntegral}
\begin{align}
& V(p(t_2)) \ = \ V(p(t_1)) + \int_{t_1}^{t_2}\langle\!\langle\nabla{V}(p(t)),\tilde{v}(t)\rangle\!\rangle\,\mathrm{d}t \label{eq:potentialEnergyIntegral:A} \allowdisplaybreaks \\
& \qquad + \sum_{\ell}\int_{t_1}^{t_2}U^{\omega}_{\ell}(t)\,\langle\!\langle\nabla{V}(p(t)),f^{\omega}_{\ell}(p(t))\rangle\!\rangle\,\mathrm{d}t, \label{eq:potentialEnergyIntegral:B}
\end{align}
\end{subequations}
and the inner product of the gradient of~$V$ and the velocity can be written as
\begin{subequations}\label{eq:mixedTermIntegral}
\begin{align}
& \langle\!\langle\nabla{V}(p(t_2)),\tilde{v}(t_2)\rangle\!\rangle \ = \ \langle\!\langle\nabla{V}(p(t_1)),\tilde{v}(t_1)\rangle\!\rangle \label{eq:mixedTermIntegral:A} \allowdisplaybreaks \\
& + \int_{t_1}^{t_2}\langle\!\langle\nabla^2{V}(p(t))\tilde{v}(t),\tilde{v}(t)\rangle\!\rangle\,\mathrm{d}t \label{eq:mixedTermIntegral:B} \allowdisplaybreaks \\
& - \int_{t_1}^{t_2}\big(\langle\!\langle\nabla{V}(p(t)),R\,\tilde{v}(t)\rangle\!\rangle + \|\nabla{V}(p(t))\|^2\big)\mathrm{d}t \label{eq:mixedTermIntegral:C} \allowdisplaybreaks \\
& + \sum_{\ell}\int_{t_1}^{t_2}U^{\omega}_{\ell}(t)\,\langle\!\langle\nabla^2{V}(p(t))\,f^{\omega}_{\ell}(p(t)),\tilde{v}(t)\rangle\!\rangle\,\mathrm{d}t \label{eq:mixedTermIntegral:D} \allowdisplaybreaks \\
& - \sum_{\ell}\int_{t_1}^{t_2}U^{\omega}_{\ell}(t)\,\langle\!\langle\nabla{V}(p(t)),R\,f^{\omega}_{\ell}(p(t))\rangle\!\rangle\,\mathrm{d}t \label{eq:mixedTermIntegral:E} \allowdisplaybreaks \\
& - \sum_{\ell}\int_{t_1}^{t_2}U^{\omega}_{\ell}(t)\,\langle\!\langle\nabla{V}(p(t)),\operatorname{D}\!f^{\omega}_{\ell}(p(t))\tilde{v}(t)\rangle\!\rangle\,\mathrm{d}t \label{eq:mixedTermIntegral:F} \allowdisplaybreaks \\
& - \sum_{\ell,\ell'}\int_{t_1}^{t_2}\!U^{\omega}_{\ell,\ell'}(t)\,\langle\!\langle\nabla{V}(p(t)),\operatorname{D}\!f^{\omega}_{\ell}(p(t))f^{\omega}_{\ell'}(p(t))\rangle\!\rangle\,\mathrm{d}t. \label{eq:mixedTermIntegral:G}
\end{align}
\end{subequations}
The sum of~\cref{eq:kineticEnergyIntegral,eq:potentialEnergyIntegral} together with the $\varepsilon$-fold of \cref{eq:mixedTermIntegral} will lead to the asserted equation~\cref{eq:integralExpansion}. To see this, note that~\cref{eq:integralExpansion:A} originates from the sum of the contributions in~\cref{eq:kineticEnergyIntegral:A}, \cref{eq:kineticEnergyIntegral:B}, and \cref{eq:potentialEnergyIntegral:A} together with the $\varepsilon$-fold of the contributions in~\cref{eq:mixedTermIntegral:A,eq:mixedTermIntegral:B,eq:mixedTermIntegral:C}. It is left to specify the definitions of $D^{\omega}_1E_{\varepsilon}$ and $D^{\omega}_2E_{\varepsilon}$ in~\cref{eq:integralExpansion:B}. To get a factor~$\frac{1}{\omega}$, we apply integration by parts to the integrals in~\cref{eq:kineticEnergyIntegral:C,eq:kineticEnergyIntegral:D,eq:kineticEnergyIntegral:E,eq:potentialEnergyIntegral:B,eq:mixedTermIntegral:D,eq:mixedTermIntegral:E,eq:mixedTermIntegral:F,eq:mixedTermIntegral:G}. For example, integration by parts in~\cref{eq:potentialEnergyIntegral:B} leads to
\begin{subequations}\label{eq:potentialEnergyIntegral:B:list}
\begin{align}
& \sum_{\ell}\int_{t_1}^{t_2}U^{\omega}_{\ell}(t)\,\langle\!\langle\nabla{V}(p(t)),f^{\omega}_{\ell}(p(t))\rangle\!\rangle\,\mathrm{d}t \ = \ \\
& \hphantom{-} \Big[\sum_{\ell}\hat{U}^{\omega}_{\ell}(t)\,\langle\!\langle\nabla{V}(p(t)),f^{\omega}_{\ell}(p(t))\rangle\!\rangle\Big]_{s=t_1}^{s=t_2} \label{eq:potentialEnergyIntegral:B:A} \allowdisplaybreaks \\
& - \sum_{\ell}\int_{t_1}^{t_2}\hat{U}^{\omega}_{\ell}(t)\,\langle\!\langle\nabla^2{V}(p(t))\,\tilde{v}(t),f^{\omega}_{\ell}(p(t))\rangle\!\rangle\,\mathrm{d}t \label{eq:potentialEnergyIntegral:B:B} \allowdisplaybreaks \\
& - \sum_{\ell,\ell'}\int_{t_1}^{t_2}\hat{U}^{\omega}_{\ell,\ell'}(t)\,\langle\!\langle\nabla^2{V}(p(t))\,f^{\omega}_{\ell}(p(t)),f^{\omega}_{\ell'}(p(t))\rangle\!\rangle\,\mathrm{d}t \label{eq:potentialEnergyIntegral:B:C} \allowdisplaybreaks \\
& - \sum_{\ell}\int_{t_1}^{t_2}\hat{U}^{\omega}_{\ell}(t)\,\langle\!\langle\nabla{V}(p(t)),\operatorname{D}\!f^{\omega}_{\ell}(p(t))\tilde{v}(t)\rangle\!\rangle\,\mathrm{d}t \label{eq:potentialEnergyIntegral:B:D} \allowdisplaybreaks \\
& - \sum_{\ell,\ell'}\int_{t_1}^{t_2}\!\hat{U}^{\omega}_{\ell',\ell}(t)\,\langle\!\langle\nabla{V}(p(t)),\operatorname{D}\!f^{\omega}_{\ell}(p(t))f^{\omega}_{\ell'}(p(t))\rangle\!\rangle\,\mathrm{d}t, \label{eq:potentialEnergyIntegral:B:E}
\end{align}
\end{subequations}
where the antiderivative $\hat{U}^{\omega}_{\ell}\colon\mathbb{R}\to\mathbb{R}$ of $U^{\omega}_{\ell}$ is defined by
\[
\hat{U}^{\omega}_{\ell}(t) \ := \ -\frac{2}{\omega\,\omega_{\ell}}\,\cos(\omega\,\omega_{\ell}\,s + \varphi_{\ell})
\]
and the function $\hat{U}^{\omega}_{\ell,\ell'}\colon\mathbb{R}\to\mathbb{R}$ is defined by
\[
\hat{U}^{\omega}_{\ell,\ell'}(t) \ := \ U^{\omega}_{\ell}(t)\,\hat{U}^{\omega}_{\ell'}(t).
\]
Note that both $\hat{U}^{\omega}_{\ell}$ and $\hat{U}^{\omega}_{\ell,\ell'}$ contain the desired factor~$\frac{1}{\omega}$. The negative of the expression in the square brackets in~\cref{eq:potentialEnergyIntegral:B:A} contributes to $D^{\omega}_1E_{\varepsilon}$. The integrands in the remaining terms in~\cref{eq:potentialEnergyIntegral:B:B,eq:potentialEnergyIntegral:B:C,eq:potentialEnergyIntegral:B:D,eq:potentialEnergyIntegral:B:E} contribute to $D^{\omega}_2E_{\varepsilon}$. When we apply the same procedure to~\cref{eq:kineticEnergyIntegral:C,eq:kineticEnergyIntegral:D,eq:kineticEnergyIntegral:E,eq:potentialEnergyIntegral:B,eq:mixedTermIntegral:D,eq:mixedTermIntegral:E,eq:mixedTermIntegral:F,eq:mixedTermIntegral:G}, we obtain all contributions to the functions $D^{\omega}_1E_{\varepsilon}$ and $D^{\omega}_2E_{\varepsilon}$ in~\cref{eq:integralExpansion}.

It remains to prove the asserted estimates for the above functions $D^{\omega}_1E_{\varepsilon}$ and $D^{\omega}_2E_{\varepsilon}$. For this purpose, we derive suitable estimates for their constituents. It follows directly from the definitions of the $\hat{U}^{\omega}_{\ell}$ and the $\hat{U}^{\omega}_{\ell,\ell'}$ that there exist $c_1,c_2>0$ such that
\begin{equation}\label{eq:UEllEstimates}
\big|\hat{U}^{\omega}_{\ell}(t)\big| \ \leq \ \frac{c_1}{\omega} \qquad \text{and} \qquad \big|\hat{U}^{\omega}_{\ell,\ell'}(t)\big| \ \leq \ \frac{c_2}{\omega}    
\end{equation}
for every $\omega>0$, all $\ell,\ell'\in\Lambda$, and every $t\in\mathbb{R}$. Fix an arbitrary $L>0$. We conclude from Proposition~\ref{thm:11112017} that the derivatives of~$V$ are bounded by constants on the $L$-sublevel set of~$V$. It also follows from Proposition~\ref{thm:11112017} that there exist $\alpha_1,\alpha_2>0$ such that
\[
\|\nabla{V_i(p)}\|^2 \ \leq \ \alpha_1\,V_i(p) \quad \text{and} \quad \|\nabla^2V_i(p)\| \ \leq \ \alpha_2
\]
for every $i\in\{1,\ldots,N\}$ and every $p\in\mathbb{R}^{nN}$ with $V(p)\leq{L}$, where $\|\nabla^2V_i(p)\|$ denotes the induced matrix norm of $\nabla^2V_i(p)$. A direct computation of the derivatives of the vector fields~$f^{\omega}_{\ell}$ in~\cref{eq:newVectorFields} followed by an application of the above estimates for $\nabla{V_i(p)}$ and $\nabla^2{V_i(p)}$ shows that there exist $\beta_0,\beta_1,\beta_2>0$ such that
\begin{subequations}\label{eq:badEstimates}
\begin{align}
\|f^{\omega}_{\ell}(p)\| & \ \leq \ \beta_0, \label{eq:badEstimates:A} \allowdisplaybreaks \\
\|\operatorname{D}\!f^{\omega}_{\ell}(p)v\| & \ \leq \ \beta_1\,\|v\| , \label{eq:badEstimates:B} \allowdisplaybreaks \\
\|\operatorname{D}^2\!f^{\omega}_{\ell}(p)(v,w)\| & \ \leq \ \sqrt{\omega}\,\beta_2\,\|v\|\,\|w\| \label{eq:badEstimates:C}
\end{align}
\end{subequations}
for every $\omega\geq1$, every $\ell\in\Lambda$, every $p\in\mathbb{R}^{nN}$ with $V(p)\leq{L}$, and all $v,w\in\mathbb{R}^{nN}$. Now we have estimates for all constituents of $D^{\omega}_1E_{\varepsilon}$ and $D^{\omega}_2E_{\varepsilon}$. Except for the terms in~\cref{eq:UEllEstimates} and \cref{eq:badEstimates:C}, all other constituents of $D^{\omega}_1E_{\varepsilon}$ and $D^{\omega}_2E_{\varepsilon}$ are bounded by a constant on the $L$-sublevel set of~$E_0$, uniformly in~$\omega$. Note that the second derivatives of the vector fields~$f^{\omega}_{\ell}$ in \cref{eq:badEstimates:C} only contribute to $D^{\omega}_2E_{\varepsilon}$ but not to $D^{\omega}_1E_{\varepsilon}$. For this reason, we get a factor~$\frac{1}{\omega}$ in the estimate for $D^{\omega}_1E_{\varepsilon}$ and a factor $\frac{1}{\sqrt{\omega}}$ in the estimate for $D^{\omega}_2E_{\varepsilon}$.
\end{proof}

\subsection{Proof of practical exponential stability}
To complete the proof of Theorem~\ref{thm:mainResult}, suppose that for every point~$p$ of~\cref{eq:setOfDesiredFormations}, the framework $\mathcal{G}(p)$ is infinitesimally rigid. Then, there exist $\varepsilon,L,\gamma_2>0$ as in Lemma~\ref{thm:ChetaevTrick}. Next, for every $\omega>0$, choose smooth real-valued functions $D^{\omega}_1E_{\varepsilon},D^{\omega}_2E_{\varepsilon}$ on $\mathbb{R}^{2nN}$ as in Proposition~\ref{thm:averaging}. Using a standard comparison lemma for ordinary differential equations, we conclude from~\cref{eq:integralExpansion} that, for every $\omega>0$, every solution $\tilde{x}=\colon{I}\to\mathbb{R}^{2nN}$ of~\cref{eq:transformedControlAffineForm}, and all $t_2>t_2$ in $I$, the following implication holds: if $E(\tilde{x}(t))\leq{L}$ for every $t\in[t_1,t_2]$, then
\[
m(t_2) \ \leq \ m(t_1)\,\mathrm{e}^{-\gamma_2(t_2-t_1)} + \int_{t_1}^{t_2}\mathrm{e}^{-\gamma_2(t_2-t)}\,u(t)\,\mathrm{d}t
\]
with the abbreviations
\begin{align*}
m(t) & \ := \ E_{\varepsilon}(\tilde{x}(t))+(D^{\omega}_1E_{\varepsilon})(t,\tilde{x}(t)), \\
u(t) & \ := \ \gamma_2\,(D^{\omega}_1E_{\varepsilon})(t,\tilde{x}(t)) + (D^{\omega}_2E_{\varepsilon})(t,\tilde{x}(t)).
\end{align*}
Since $D^{\omega}_1E_{\varepsilon}$ and $D^{\omega}_2E_{\varepsilon}$ satisfy the estimates in Proposition~\ref{thm:averaging}, this implies, after possibly shrinking $L>0$, that there exist $\omega_0,\rho,\mu,\lambda>0$ such that for every $\omega\geq\omega_0$, every $t_0\in\mathbb{R}$, and every $\tilde{x}_0\in\mathbb{R}^{2nN}$ with $E(\tilde{x}_0)\leq{L}$, the maximal solution $\tilde{x}=(p^{\top},\tilde{v}^{\top})$ of~\cref{eq:transformedControlAffineForm} with $\tilde{x}(t_0)=\tilde{x}_0$ satisfies
\begin{align*}
E_{\varepsilon}(\tilde{x}(t)) \ \leq \ \sqrt{\frac{\rho}{\omega}} + \lambda\,E_{\varepsilon}(\tilde{x}_0)\,\mathrm{e}^{-\mu(t-t_0)}
\end{align*}
for every $t\geq{t_0}$. Because of Lemma~\ref{thm:ChetaevTrick}, an estimate of the above form, but with possibly different constants $\omega_0,\rho,L,\lambda$, also holds for the total energy~$E=E_0$ along the solutions of~\cref{eq:transformedControlAffineForm}. Now, by adjusting again the constants $\omega_0,\rho,L,\lambda$, the claim follows from Lemma~\ref{thm:transformedEnergy}.
\end{document}